\documentclass[11pt,reqno]{amsart}
\usepackage{hyperref}
\selectfont

\usepackage{enumerate}

\numberwithin{equation}{section}
 \theoremstyle{plain}
\newtheorem{theorem}{Theorem}[section]
\newtheorem{lemma}[theorem]{Lemma}

\newtheorem{definition}[theorem]{Definition}

\newcommand{\no}{\nonumber}

\begin{document}

\begin{center}
{\Large\bf Congruences modulo powers of $5$ for odd ranks}
\end{center}

\begin{center}

Renrong Mao  and Zhiqian Zhou

Department of Mathematics,\\
 Soochow University, \\
 Suzhou, 215006,
People's Republic of China\\[6pt]

Email: rrmao@suda.edu.cn,
  20217907002@stu.suda.edu.cn

\end{center}

\noindent {\bf Abstract.} 	In 2007, Andrews studied the odd Durfee symbols and their odd ranks.
Let $N^0(m,k,n)$ denote the number of odd Durfee symbols of $n$ with odd rank congruent to $m$ modulo $k$. 
Motivated by Andrews' work, many authors obtained generating functions of $N^0(m,k,n)$ from which relations between odd ranks are proved.
In this paper, we establish a family of congruences for odd ranks modulo powers of $5$.



        \allowdisplaybreaks

	\section{Introduction}
A partition of a positive integer $n$ is a sequence of non-increasing positive integers whose sum equals $n$.
Ramanujan \cite{ra1} proved the three famous congruences:
\begin{align}
	p(5n+4)&\equiv 0\pmod 5\label {1i1}\\
	p(7n+5)&\equiv 0\pmod 7\label {1i2}\\
	p(11n+6)&\equiv 0\pmod{11}\no,
\end{align}
where 
$p(n)$ denote the number of partitions of $n$.
Moreover, he also conjectured that, if $\alpha\ge 1$ and $\delta_\alpha$ is the reciprocal modulo $5^\alpha$ of 24,  then 
\begin{align}
	p(5^\alpha +\delta_\alpha) \equiv 0\pmod {5^\alpha}.\label{pnmod5}
\end{align}
G. N. Watson \cite{watson1938ramanujans} first proved Ramanujan's conjecture by applying the modular equation of fifth order. 
In \cite{hh},  M. D. Hirschhorn and D. C. Hunt provided a more elementary proof of \eqref{pnmod5}. On the other hand, to provide combinatorial interpretations for Ramanujan’s congruences \eqref{1i1}-\eqref{1i2},
Dyson \cite{dyson} defined the rank of a partition to be the largest part minus the number
of parts. Denote the number of partitions of $n$ with rank congruent to $m$ modulo $k$ by $N(m,k,n)$. Then he conjectured that
\begin{align*}
	N(m,5,5n+4)&=\frac{p(5n+4)}{5}
	\intertext{for all $0\leq m\leq 4$ and for $0\leq j\leq 6$}
	N(j,7,7n+5)&=\frac{p(7n+5)}{7}.
\end{align*}
Dyson's conjecture was proved 
by Atkin and Swinnerton-Dyer \cite{atkin}.
Recently, D. Chen, R. Chen and Garvan \cite{cc} established congruences for ranks of partitions.
Let $$ a_f(n):=N(0,2,n)-N(1,2,n).$$ Then they proved that,
for all $\alpha \geq 3$ and all $n \geq 0$,
$$
a_f\left(5^\alpha n+\delta_\alpha\right)+a_f\left(5^{\alpha-2} n+\delta_{\alpha-2}\right) \equiv 0 \quad\left(\bmod 5^{\left\lfloor\frac{1}{2} \alpha\right\rfloor}\right),
$$
where $\delta_\alpha$ satisfies $0<\delta_\alpha<5^\alpha$ and $24 \delta_\alpha \equiv 1\left(\bmod 5^\alpha\right)$.
In this paper, we obtain congruences for odd ranks.

To study the partition-theoretic interpretation of Watson's the third order mock theta function $\omega(q)$ \cite{wat}, Andrews \cite{andrews} defined the odd Durfee symbol.

\allowdisplaybreaks
\begin{definition}\label{defodd}
	An odd Durfee symbol of $n$ is a two-rowed array with a subscript of the form
	\begin{align*}
		\begin{pmatrix}
			a_{1}& a_{2}& \cdots& a_{s}\\
			b_{1}& b_{2}& \cdots& b_{t}
		\end{pmatrix}_{D}
	\end{align*}
	wherein all entries are odd numbers such that
	\begin{enumerate}[(1)]
		\item $2D+1\geq a_{1}\geq a_{2}\geq \cdots \geq a_{s}\ge 0$;
		\item $2D+1\geq b_{1}\geq b_{2}\geq \cdots \geq b_{t}\ge 0$;		
		\item $n=\sum_{i=1}^{s}a_{i}+\sum_{j=1}^{t}b_{j}+2D^{2}+2D+1$.
	\end{enumerate}
\end{definition}
We remark that
Definition \ref{defodd} is made by Ji \cite{ji} and is equivalent to Andrews' original definition.
Andrews also defined the odd rank of an odd Durfee symbol
to be the number of entries in the top row minus the number of entries in the bottom row. 
Motivated by Andrews' work, odd ranks are widely studied by many authors.
Let $N^0(m,k,n)$ denote the number of odd Durfee symbols of $n$ with odd rank congruent to $m$ modulo $k$.
Wang \cite{wang1} established identities between odd ranks modulo $8$. For example, he proved that
\begin{equation*}
	\sum_{n=0}^{\infty}(N^0(0,8,8n+1)-N^0(4,8,8n+1))q^n=\frac{(q^2;q^{2})^4_{\infty}}{(q;q)^2_{\infty}(q^4;q^{4})_{\infty}}.
\end{equation*}
In the above equation and for the rest of this article, we use the notations
\begin{align*}
	(x)_{\infty}:&=(x ; q)_{\infty}:=\prod_{k=0}^{\infty}\left(1-x q^{k}\right) \\
	\left(x_{1}, \ldots, x_{m}\right)_{\infty}:&=\left(x_{1}, \ldots, x_{m} ; q\right)_{\infty}:=\left(x_{1} ; q\right)_{n} \cdots\left(x_{m} ; q\right)_{\infty} \\
	\left[x_{1}, \ldots, x_{m}\right]_{\infty}:&=\left[x_{1}, \ldots, x_{m} ; q\right]_{\infty}:=\left(x_{1}, q / x_{1}, \ldots, x_{n}, q / x_{n} ; q\right)_{\infty},
\end{align*}
where $q=e^{2\pi i\tau}$ with $\text{Im}(\tau)>0$.

Moreover, Wang also obtained that, for $n\geq0$,
\begin{align}
	N^0(0,8 ; 8 n+r)&=N^o(4,8 ; 8 n+r), & r \in\{5,7\} \no
	\intertext{and}
	N^0(1,8 ; 8 n+r)&=N^o(3,8 ; 8 n+r), & r \in\{4,6\}.\no
\end{align}

Cui and Gu \cite{cui} studied the generating functions of odd ranks modulo $3$ and $6$. For example, they showed that 
\begin{align*}
	\sum_{n=0}^{\infty}(N^0(0,3,n)-N^0(1,3,n))q^n&=q\sum_{n=0}^{\infty}\frac{q^{2n(n+1)}(q;q^{2})_{n+1}}{(q^3;q^{6})_{n+1}},
	\\
	\sum_{n=0}^{\infty}(N^0(0,6,n)-N^0(3,6,n))q^n&=q\frac{(q^3;q^{3})^2_{\infty}(q^{12};q^{12})^2_{\infty}}{(q^2;q^2)_{\infty}(q^6;q^{6})^2_{\infty}}.
\end{align*}
More recently, Xia \cite{xia} considered odd ranks modulo $12$ and found the relations
\begin{align*}
	N^0(4,12,4n+1)&\geq N^0(6,12,4n+1),\\
	N^0(0,8,16n+3)-N^0(4,8,16n+3)&=N^0(0,12,4n+1)-N^0(6,12,4n+1),
\end{align*}
for $n\geq0.$ 

The main goal of this paper is the following congruences for odd rank modulo $5$.
\begin{theorem}\label{thmain}
	For integers $\alpha\geq1$, let
	\begin{align*}
		\lambda_{ \alpha} :=
		\left\{
		\begin{aligned}
			&	\frac{2 \cdot 5^{\alpha}-1}{3}, \quad&&\text{if $\alpha$ is odd,} \\
			&\frac{ 5^{\alpha}-1}{3},\quad&&\text{else}
		\end{aligned}
		\right..
	\end{align*}
	Then we have 
	\begin{align}
		N^0(1,5,5^{\alpha+1}n+5\lambda_\alpha+2)\equiv N^0(2,5,5^{\alpha+1}n+5\lambda_\alpha+2)\pmod{5^{\left\lfloor\frac{\alpha+1}{2}\right\rfloor }}.\no
	\end{align}
\end{theorem}
To establish Theorem \ref{thmain}, we first need the following generating function.
\begin{theorem}\label{thmain2}
	We have 
	\begin{align}
		&\sum_{n=0}^{\infty}(N^0(1,5,5n+2)-N^0(2,5,5n+2))q^n=\frac{(q^{5};q^{5})^2_{\infty}}{(q^2;q^{2})_{\infty}}\label{id52}.
	\end{align}
\end{theorem}
We apply the method of Atkin and Swinnerton-Dyer \cite{atkin} to prove Theorem \ref{thmain2} and one can also obtain generating functions of $N^0(1,5,5n+i)-N^0(2,5,5n+i)$ with $i=0, 1, 3, 4$ from the proof. 

We prove Theorem \ref{thmain} by similar arguments in \cite{peter,wang}. In Section \ref{s2}, we establish identities between modular functions on $\Gamma_0(10)$. They are applied to obtain Theorem \ref{thmain} whose proof is given in Section \ref{s4}. 
Theorem \ref{thmain2} is proved in Section \ref{s3}.

\section{Preliminaries}\label{s2}
%
 Dedekind's eta-function is defined by
\begin{align}
	\eta(\tau):=q^{\frac{1}{24}}\prod_{j=1}^\infty(1-q^j).\label{defeta}
\end{align}
Following \cite{wang} we also need 
\begin{align*}
	\rho&:=\rho(\tau):=\frac{\eta^2(2 \tau) \eta^4(5 \tau)}{\eta^4(\tau) \eta^2(10 \tau)},\\
	t&:=t(\tau):=\frac{\eta^2(5 \tau) \eta^2(10 \tau)}{\eta^2(\tau) \eta^2(2 \tau)} ,\\
	Z:&=Z(\tau)=\frac{\eta(50\tau)}{\eta(2\tau)}
	\intertext{and}
	H:&=H(\tau)=\frac{\eta^2(\tau)}{\eta^2(25\tau)}.
\end{align*}
Using the criteria for the modularity of eta-products \cite[Theorem 4.7]{new}, we can check that $\rho, t$ are modular functions on $\Gamma_0(10)$ and $Z, H$ are on $\Gamma_0(50)$ and $\Gamma_0(25)$, respectively.

For a function 
\begin{align*}
	g(\tau)=\sum_{n=-\infty}^{\infty} a_{g}(n) q^{n},
\end{align*}
the operator $U_{k}$ is defined by
\begin{align}
	U_{ k}(g)(\tau):&=\frac{1}{k}\sum_{\lambda=0}^{k-1}g\left(\frac{\tau+\lambda}{k}\right).\no
\end{align}
A straightforward calculation gives
\begin{align}
	U_{k}(f(q^k)g)(\tau)&=f(q) \sum_{n=-\infty}^{\infty} a_{g}(k n) q^{n}.\label{2121330}
\end{align}
Recall \cite[Lemma 2.3]{wang}:
\begin{lemma}
	Let
	\begin{align*}
		& a_0(t)=-t, \\ &
		a_1(t)=-t\left(2 \cdot 5+5^2 t\right), \\& a_2(t)=-t\left(11 \cdot 5+2 \cdot 5^3 t+5^4 t^2\right), \\
		& a_3(t)=-t\left(28 \cdot 5+11 \cdot 5^3 t+2 \cdot 5^5 t^2+5^6 t^3\right) \\
		& a_4(t)=-t\left(7 \cdot 5^2+28 \cdot 5^3 t+11 \cdot 5^5 t^2+2 \cdot 5^7 t^3+5^8 t^4\right) .
	\end{align*}
	Then, for $u: \mathbb{H} \rightarrow \mathbb{C}$ and $j \in \mathbb{Z}$, we have
	\begin{align}
		U_5\left(u t^j\right)=-\sum_{l=0}^4 a_l(t) U_5\left(u t^{j+l-5}\right)\label{eqf}.
	\end{align}
\end{lemma}

\begin{lemma}\label{legroup}
	Let
	\begin{align}
		U^{(0,j)}(f):=U_5(Z\cdot\rho^j\cdot f),\no\\
		U^{(1,j)}(f):=U_5(H\cdot\rho^j\cdot f).\no
	\end{align}
	Then we have
	\begin{align*}
		\intertext{	Group I}
		U^{(0,0)}(1)&=5t+5^2t^2-\rho(5t)\\
		U^{(0,0)}(t^{-1})&=1+5t-\rho\\
		U^{(0,0)}(t^{-2})&=-3-3\cdot 5t+\rho(4-5^2t)\\
		U^{(0,0)}(t^{-3})&=-1+4\cdot5^2t+6\cdot5^3t^2+5^5t^3+\rho(-5+4\cdot5^2t)\\
		U^{(0,0)}(t^{-4})&=63-33\cdot5^2t-17\cdot5^4t^2-18\cdot5^5t^3-5^7t^4\\&\quad+\rho(-2\cdot5^2+5^4t+2\cdot5^5t^2+2\cdot5^6t^3)
		\intertext{	Group II}
		U^{(0,1)}(1)&=-4t+5^2t^2-3\cdot5^4 t^3-4\cdot5^5t^4-5^7t^5\\&\quad+\rho(8\cdot5t+12 \cdot5^3t^2+7\cdot 5^5t^3+8\cdot5^6t^4+5^8t^5)\\
		U^{(0,1)}( t^{-1})&=5t\\
		U^{(0,1)}( t^{-2})&=1\\
		U^{(0,1)}( t^{-3})&=-7+12\cdot5t+7\cdot5^3t^2+5^5t^3+\rho(5-4\cdot5^2t-5^4t^2)\\
		U^{(0,1)}( t^{-4})&=23-19\cdot5^2t-13\cdot5^4t^2-13\cdot5^5t^3-5^7t^4\\
		&\quad+\rho(-6\cdot5+7\cdot5^3t+2\cdot5^5t^2+5^6t^3)\\
		\intertext{	Group III}
		U^{(1,0)}(1)&=-1\\
		U^{(1,0)}(t^{-1})&=t^{-1}-5+5^2t+\rho (-t^{-1}+5^2)\\
		U^{(1,0)}(t^{-2})&=-4 t^{-1}+4\cdot 5- 5^4t^2+\rho(4t^{-1}-4\cdot5^2)\\
		U^{(1,0)}(t^{-3})&=t^{-1}+2\cdot 5^3+7\cdot5^4t+9\cdot5^5t^2+4\cdot5^6t^3+\rho(-2\cdot5^3-5^5t-5^6t^2)\\
		U^{(1,0)}(t^{-4})&=16\cdot5 t^{-1}-118\cdot5^2-14\cdot5^5t-18\cdot5^6t^2-2\cdot5^8t^3-5^8t^4\\&\quad+\rho(-18\cdot5t^{-1}+38\cdot5^3+2\cdot5^6t+2\cdot5^7t^2)
		\intertext{	Group IV}
		U^{(1,1)}(1)&=-2+5\rho\\
		U^{(1,1)}( t^{-1})&=-5-5^2t+5\rho \\
		U^{(1,1)}( t^{-2})&=t^{-1}- 5-5^3t- 5^4t^2+\rho(-t^{-1}+5^2+5^3t)\\
		U^{(1,1)}( t^{-3})&=-8t^{-1}+59\cdot 5+7\cdot5^4t+8\cdot5^5t^2+4\cdot5^6t^3\\&\quad+\rho(9t^{-1}-19\cdot5^2-5^5t-4\cdot5^5t^2)\\
		U^{(1,1)}( t^{-4})&=29 t^{-1}-67\cdot5^2-42\cdot5^4t-54\cdot5^5t^2-7\cdot5^7t^3-5^8t^4\\&\quad+\rho(-7\cdot5t^{-1}+19\cdot5^3+6\cdot5^5t+6\cdot5^6t^2+5^7t^3)
	\end{align*}
\end{lemma}
\begin{proof}[Sketch of proof]	
	All of the equations in Group I-IV are identities between modular functions on $\Gamma_0(10)$.
	Use \cite[Theorem 4.8]{lig} to compute the orders of an eta-product at the cusps of $\Gamma_0(10)$.
	Armed with \cite[Theorem 4]{gk}, one can obtain the lower bounds for orders of $U_5(f)$ at cusps of $\Gamma_0(10)$ where
	$f$ is an eta-product on $\Gamma_0(50)$.
	Then applying the Valence Formula \cite[p.98]{rank}, we can prove each identity in Group I-IV by verifying the $q$-expansions of both sides agree up to some power of $q$.
	All of the above calculations can be done by the MAPLE package ETA \cite{maple1}.
	For example, the Maple commands of verifying the second identity in Group I are provided at https://github.com/dongpanghu/code3.
\end{proof}
Following \cite{peter}, we call a map $d: \mathbb{Z} \times \mathbb{Z} \longrightarrow \mathbb{Z}$ a discrete array if for each $i$ the map $d(i,-): \mathbb{Z} \longrightarrow \mathbb{Z}$, by $j \mapsto d(i, j)$ has finite support.
\begin{lemma}\label{leuab}
	There exists discrete arrays $a_{i,j}, b_{i,j}$ with $0\leq i, j\leq1$ such that for $k\in\mathbb{Z}$
	\begin{align}\label{uab}
		U^{(i,j)}\left( t^k\right)= \sum_{n=\left\lceil\frac{k-s_{i,j}}{5}\right\rceil}^{\infty} a_{i,j}(k, n) t^n+\rho\left(\sum_{n=\left\lceil\frac{k-s_{i,j}}{5}\right\rceil}^{\infty} b_{i,j}(k, n) t^n\right),
	\end{align}
	where 
	\begin{align*}
		s_{i,j}=\left\{
		\begin{aligned}
			-&1,\quad\text{when}\quad (i,j)=(0,0)\\
			-	&2,\quad\text{when}\quad (i,j)=(0,1)\\
			&4,\quad\text{when}\quad (i,j)=(1,0)\\
			&3,\quad\text{when}\quad (i,j)=(1,1)
		\end{aligned}
		\right..
	\end{align*}
	Moreover, the values of $a_{i,j}(k, n)$ and $b_{i,j}(k, n)$ for $-4 \leq k\leq 0$ are given in Group I--IV of Lemma \ref{legroup}, and for other $k, a_{i,j}(k, n), b_{i,j}(k, n)$ satisfy the recurrence relation
	in \cite[Eq. (2.17)]{wang}$:$
	\begin{align}\label{rr}
		m(	k, 	n)= & \left(7 \cdot 5^2 m(	k-1, 	n-1)+28 \cdot 5^3 m(	k-1, 	n-2)+11 \cdot 5^5 m(	k-1, 	n-3)\right. \no\\
		& \left.+2 \cdot 5^7 m(	k-1, 	n-4)+5^8 m(	k-1, 	n-5)\right)+(28 \cdot 5 m(	k-2, 	n-1)\no \\
		& \left.+11 \cdot 5^3 m(	k-2, 	n-2)+2 \cdot 5^5 m(	k-2, 	n-3)+5^6 m(	k-2, 	n-4)\right) \no\\
		& +\left(11 \cdot 5 m(	k-3, 	n-1)+2 \cdot 5^3 m(	k-3, 	n-2)+5^4 m(	k-3, 	n-3)\right) \no\\
		& +\left(2 \cdot 5 m(	k-4, 	n-1)+5^2 m(	k-4, 	n-2)\right)+m(	k-5, 	n-1) .
	\end{align}
\end{lemma}
\begin{proof}
	Applying Group I-IV in Lemma \ref{legroup}, we verify that the result holds for $-4 \leq k\leq 0$. Then one can prove Lemma \ref{leuab} by \eqref{eqf} and induction on $k$.
\end{proof}
\section{Proof of Theorem \ref{thmain2}}\label{s3}
We need two lemmas.
\begin{lemma}\label{lepp}
	Let
	\begin{align*}
		P_0&:=\frac{[q^{10},q^{20};q^{50}]_{\infty}(q^{50};q^{50})^2_{\infty}}{[q^{5},q^{5},q^{15};q^{50}]_{\infty}},\\
		P_3&:=\frac{q^3[q^{10},q^{20};q^{50}]_{\infty}(q^{50};q^{50})^2_{\infty}}{[q^{5},q^{15},q^{15};q^{50}]_{\infty}},\\
		P_4&:=\frac{q^4[q^{10},q^{20};q^{50}]_{\infty}(q^{50};q^{50})^2_{\infty}}{[q^{5},q^{15},q^{25};q^{50}]_{\infty}}.
	\end{align*}
	Then we have
	\begin{align}
		\sum_{n=-\infty}^{\infty}\frac{(-1)^nq^{3n^2+7n+3}}{1-q^{10n+5}}&=\frac{(q^{2};q^{2})_{\infty}}{(q^{50};q^{50})_{\infty}}\sum_{n=-\infty}^{\infty}\frac{(-1)^nq^{75n^2+125n+49}}{1-q^{50n+35}}+P_3+P_4,\label{les1}\\
		\sum_{n=-\infty}^{\infty}\frac{(-1)^nq^{3n^2+5n+2}}{1-q^{10n+5}}&=-\frac{(q^{2};q^{2})_{\infty}}{(q^{50};q^{50})_{\infty}}\sum_{n=-\infty}^{\infty}\frac{(-1)^nq^{75n^2+25n}}{1-q^{50n+5}}+P_0-P_4.\label{les2}
	\end{align}
\end{lemma}
\begin{proof}
	We only give the detailed proof of \eqref{les1}. One can obtain \eqref{les2} with a completely similar argument.
	
	Splitting the series on the right side of \eqref{les1} into five series according to the summation
	index $n$ modulo $5$, we obtain
	\begin{align}\label{s15}
		&\sum_{n=-\infty}^{\infty}\frac{(-1)^nq^{3n^2+7n+3}}{1-q^{10n+5}}\nonumber\\&=\sum_{r=-1}^3(-1)^r\sum_{n=-\infty}^{\infty}\frac{(-1)^nq^{3(5n+r)^2+7(5n+r)+3}}{1-q^{10(5n+r)+5}}\nonumber\\
		&	=-\sum_{n=-\infty}^{\infty}\frac{(-1)^nq^{75n^2+5n-1}}{1-q^{50n-5}}+\sum_{n=-\infty}^{\infty}\frac{(-1)^nq^{75n^2+35n+3}}{1-q^{50n+5}}-\sum_{n=-\infty}^{\infty}\frac{(-1)^nq^{75n^2+65n+13}}{1-q^{50n+15}}\nonumber\\
		&\quad	+\sum_{n=-\infty}^{\infty}\frac{(-1)^nq^{75n^2+95n+29}}{1-q^{50n+25}}-\sum_{n=-\infty}^{\infty}\frac{(-1)^nq^{75n^2+125n+51}}{1-q^{50n+35}}\nonumber\\
		&:=-S_1+S_2-S_3+S_4-S_5.
	\end{align}
	Set $r=0, s=3$ in \cite[Theorem 2.1]{chan} to obtain
	\begin{align}
		\frac{(q)_{\infty}^{2}}{[b_{1},b_{2},b_{3}]_{\infty}}&=
		\frac{1}{[b_{2}/b_{1},b_{3}/b_{1}]_{\infty}}\sum_{k=-\infty}^{\infty}\frac{(-1)^{k}q^{3k(k+1)/2}}{1-b_{1}q^{k}}\left(\frac{b_{1}^{2}}{b_{2}b_{3}}\right)^{k}\nonumber\\
		&\quad+\frac{1}{[b_{1}/b_{2},b_{3}/b_{2}]_{\infty}}\sum_{k=-\infty}^{\infty}\frac{(-1)^{k}q^{3k(k+1)/2}}{1-b_{2}q^{k}}\left(\frac{b_{2}^{2}}{b_{1}b_{3}}\right)^{k}\nonumber\\
		&\quad+\frac{1}{[b_{1}/b_{3},b_{2}/b_{3}]_{\infty}}\sum_{k=-\infty}^{\infty}\frac{(-1)^{k}q^{3k(k+1)/2}}{1-b_{3}q^{k}}\left(\frac{b_{3}^{2}}{b_{1}b_{2}}\right)^{k}.\label{eqchan}
	\end{align}
	Replacing $(q, b_1, b_2, b_3)$ by $(q^{50}, q^{-5}, q^{25}, q^{35})$ in \eqref{eqchan} and simplifying, we find that
	\begin{align*}
		\frac{(q^{50};q^{50})^2_{\infty}}{[q^{-5},q^{25},q^{35};q^{50}]_{\infty}}&=\frac{1}{[q^{30},q^{40};q^{50}]_{\infty}}\sum_{n=-\infty}^{\infty}\frac{(-1)^nq^{75n^2+5n}}{1-q^{50n-5}}
		\\&\quad+\frac{1}{[q^{-30},q^{10};q^{50}]_{\infty}}\sum_{n=-\infty}^{\infty}\frac{(-1)^nq^{75n^2+95n}}{1-q^{50n+25}}\\&\quad+\frac{1}{[q^{-10},q^{-40};q^{50}]_{\infty}}\sum_{n=-\infty}^{\infty}\frac{(-1)^nq^{75n^2+125n}}{1-q^{50n+35}},
	\end{align*}
	which is equivalent to
	\begin{equation}
		S_4-S_1=\frac{[q^{20};q^{50}]_{\infty}}{q^2[q^{10};q^{50}]_{\infty}}S_5+P_4.\label{s41}
	\end{equation}
	Similarly, we can apply \eqref{eqchan} to prove that
	\begin{equation}
		S_3-S_2=\frac{q^2[q^{10};q^{50}]_{\infty}}{[q^{20};q^{50}]_{\infty}}S_5-P_3.\label{s32}
	\end{equation}
	Substituting \eqref{s41} and \eqref{s32} into \eqref{s15} gives
	\begin{align}
		&\sum_{n=-\infty}^{\infty}\frac{(-1)^nq^{3n^2+7n+3}}{1-q^{10n+5}}\nonumber\\&
		=-\left(1-\frac{[q^{20};q^{50}]_{\infty}}{q^2[q^{10};q^{50}]_{\infty}}+\frac{q^2[q^{10};q^{50}]_{\infty}}{[q^{20};q^{50}]_{\infty}}\right)\times S_5+P_0+P_4\label{p04}
	\end{align}
	Replacing $q$ by $q^2$ in \cite[Eq. (2.12)]{mao} gives
	\begin{align}
		(q^2;q^2)_{\infty}
		&=\frac{[q^{20};q^{50}]_{\infty}(q^{50};q^{50})_{\infty}}{[q^{10};q^{50}]_{\infty}}-q^{2}(q^{50};q^{50})_{\infty}-\frac{q^{4}[q^{10};q^{50}]_{\infty}(q^{50};q^{50})_{\infty}}{[q^{20};q^{50}]_{\infty}},\label{mao}
	\end{align}
	which implies
	\begin{align}
		1-\frac{[q^{20};q^{50}]_{\infty}}{q^2[q^{10};q^{50}]_{\infty}}+\frac{q^2[q^{10};q^{50}]_{\infty}}{[q^{20};q^{50}]_{\infty}}=\frac{-	(q^2;q^2)_{\infty}}{q^{2}(q^{50};q^{50})_{\infty}}\label{3disp04}.
	\end{align}
	Substituting \eqref{3disp04} into \eqref{p04}, we obtain \eqref{les1}.
\end{proof}
\begin{lemma}\label{leaa}
	Let 
	\begin{align*}
		A_0:&=\frac{[q^{15},q^{15};q^{50}]_{\infty}(q^{50};q^{50})_{\infty}}{[q^{5},q^{10},q^{20};q^{50}]_{\infty}},\\
		A_1:&=\frac{q^6[q^{5};q^{50}]_{\infty}(q^{50};q^{50})_{\infty}}{[q^{10},q^{20};q^{50}]_{\infty}},\\
		A_2:&=\frac{q^2[q^{25};q^{50}]_{\infty}(q^{50};q^{50})_{\infty}}{[q^{10},q^{20};q^{50}]_{\infty}},\\
		A_3:&=\frac{q^3[q^{15};q^{50}]_{\infty}(q^{50};q^{50})_{\infty}}{[q^{10},q^{20};q^{50}]_{\infty}},\\
		A_4:&=-\frac{q^9[q^{5},q^{5};q^{50}]_{\infty}(q^{50};q^{50})_{\infty}}{[q^{10},q^{15},q^{20};q^{50}]_{\infty}}.
	\end{align*}
	Then we have 
	\begin{equation}\label{leid5}
		\frac{1}{(q^{2};q^{2})_{\infty}}\left\{P_0-P_3-2P_4\right\}=A_0+A_1+A_2+A_3+A_4.
	\end{equation}
\end{lemma}
\begin{proof}
	Equation \eqref{leid5} is equivalent to
	\begin{equation}
		P_0-P_3-2P_4=(q^{2};q^{2})_{\infty}(A_0+A_1+A_2+A_3+A_4).\label{pa}
	\end{equation}
	Invoking \eqref{mao} on the right hand side of \eqref{pa} and comparing both sides according to the powers of
	$q$ modulo $5$, we find that it suffices to show the following fives identities:
	\begin{align}
		\frac{[q^{15},q^{15};q^{50}]_{\infty}}{[q^{5},q^{10},q^{10};q^{50}]_{\infty}}+\frac{q^5[q^{15};q^{50}]_{\infty}}{[q^{10},q^{20};q^{50}]_{\infty}}-\frac{q^{10}[q^{5};q^{50}]_{\infty}}{[q^{20},q^{20};q^{50}]_{\infty}}
		&=\frac{[q^{10},q^{20};q^{50}]_{\infty}}{[q^{5},q^{5},q^{15};q^{50}]_{\infty}},\label{id511}\\
		\frac{q^6[q^{5};q^{50}]_{\infty}}{[q^{10},q^{10};q^{50}]_{\infty}}+\frac{q^{11}[q^{5},q^{5};q^{50}]_{\infty}}{[q^{10},q^{15},q^{20};q^{50}]_{\infty}}-\frac{q^6[q^{25};q^{50}]_{\infty}}{[q^{20},q^{20};q^{50}]_{\infty}}&=0,\no\\
		\frac{q^{2}[q^{25};q^{50}]_{\infty}}{[q^{10},q^{10};q^{50}]_{\infty}}-\frac{q^{2}[q^{15},q^{15};q^{50}]_{\infty}}{[q^{5},q^{10},q^{20};q^{50}]_{\infty}}+\frac{q^7[q^{15};q^{50}]_{\infty}}{[q^{20},q^{20};q^{50}]_{\infty}}&=0,\no
		\\
		\frac{q^3[q^{15};q^{50}]_{\infty}}{[q^{10},q^{10};q^{50}]_{\infty}}+\frac{q^8[q^{5};q^{50}]_{\infty}}{[q^{10},q^{20};q^{50}]_{\infty}}-\frac{q^{13}[q^{5},q^{5};q^{50}]_{\infty}}{[q^{15},q^{20},q^{20};q^{50}]_{\infty}}&=	\frac{q^{3}[q^{10},q^{20};q^{50}]_{\infty}}{[q^{5},q^{15},q^{15};q^{50}]_{\infty}},
		\no\\
		\frac{q^{9}[q^{5},q^{5};q^{50}]_{\infty}}{[q^{10},q^{10},q^{15};q^{50}]_{\infty}}+\frac{q^{4}[q^{25},;q^{50}]_{\infty}}{[q^{10},q^{20};q^{50}]_{\infty}}+\frac{q^{4}[q^{15},q^{15};q^{50}]_{\infty}}{[q^{5},q^{20},q^{20};q^{50}]_{\infty}}&=2\frac{q^{4}[q^{10},q^{20};q^{50}]_{\infty}}{[q^{5},q^{15},q^{25};q^{50}]_{\infty}}.\no
	\end{align}
	The proofs of the above fives equations are similar to each others, we only show \eqref{id511}. 
	After multiplying by $[q^{5},q^{10},q^{10},q^{20},q^{20};q^{50}]_{\infty}$ throughout and rearranging
	we find that \eqref{id511} is equivalent to
	\begin{align}
		&	q^{10}[q^{5},q^{5},q^{10},q^{10};q^{50}]_{\infty}-[q^{15},q^{15},q^{20},q^{20};q^{50}]_{\infty}\no\\&=q^5[q^5,q^{10},q^{15},q^{20};q^{50}]_{\infty}-\frac{[q^{10},q^{10},q^{10},q^{20},q^{20},q^{20};q^{50}]_{\infty}}{[q^{5},q^{15};q^{50}]_{\infty}}.\label{pp51}
	\end{align}
	With $q$ replaced by $q^{50}$, setting $(A,b,c,d,e)=(q^{45},q^{10},q^{30},q^{25},q^{25})$ in \cite[Eq. (3.1)]{chu},
	\begin{align}
		[A/b,A/c,A/d,A/e;q]_\infty-[b,c,d,e;q]_\infty=b[A,A/bc,A/bd,A/be;q]_\infty,\label{idchu}
	\end{align}
	we obtain
	\begin{equation}
		[q^{15},q^{15},q^{20},q^{20};q^{50}]_{\infty}-[q^{10},q^{20},q^{25},q^{25};q^{50}]_{\infty}	=q^{10}[q^{5},q^{5},q^{10},q^{10};q^{50}]_{\infty}.\label{pp52}
	\end{equation}
	Substituting \eqref{pp52} into \eqref{pp51}, multiplying by $\frac{[q^{5},q^{15};q^{50}]_{\infty}}{[q^{10},q^{20};q^{50}]_{\infty}}$ throughout and rearranging, we find that it suffices to show 
	\begin{align}
		&	[q^{10},q^{10},q^{20},q^{20};q^{50}]_{\infty}-[q^{5},q^{15},q^{25},q^{25};q^{50}]_{\infty}=q^5[q^5,q^{5},q^{15},q^{15};q^{50}]_{\infty},\no
	\end{align}
	which is \eqref{idchu} with
	$(q,A,b,c,d,e)$ replaced by $(q^{50},q^{45},q^{5},q^{35},q^{25},q^{25})$.
\end{proof}
Now we are in a position to prove Theorem \ref{thmain2}.
\begin{proof}[Proof of Theorem \ref{thmain2}]
	Recall \cite[Theorem 1.3]{cui}
	\begin{align*}
		&	\sum_{n=0}^{\infty}N^{0}(t,s;n)q^{n}=\frac{1}{(q^{2};q^{2})_{\infty}}\sum_{n=-\infty}^{\infty}\frac{(-1)^{n}q^{3n^{2}+3n+1+(2n+1)t}}{1-q^{(2n+1)s}}
		,
	\end{align*}
	which gives
	\begin{align}
		\no	&\sum_{n=0}^{\infty}(N^0(1,5,n)-N^0(2,5,n))q^n \\
		=&\frac{1}{(q^{2};q^{2})_{\infty}}\sum_{n=-\infty}^{\infty}\frac{(-1)^nq^{3n^2+5n+2}}{1-q^{10n+5}}-\frac{1}{(q^{2};q^{2})_{\infty}}\sum_{n=-\infty}^{\infty}\frac{(-1)^nq^{3n^2+7n+3}}{1-q^{10n+5}}\no\\
		=&\frac{1}{(q^{50};q^{50})_{\infty}}\sum_{n=-\infty}^{\infty}\frac{(-1)^nq^{75n^2+25n}}{1-q^{50n+5}}-\frac{1}{(q^{50};q^{50})_{\infty}}\sum_{n=-\infty}^{\infty}\frac{(-1)^nq^{75n^2+125n+49}}{1-q^{50n+35}}\no\\
		&+\frac{1}{(q^{2};q^{2})_{\infty}}(P_0-P_3-2P_4)\no\\
		=&\frac{1}{(q^{50};q^{50})_{\infty}}\sum_{n=-\infty}^{\infty}\frac{(-1)^nq^{75n^2+25n}}{1-q^{50n+5}}-\frac{1}{(q^{50};q^{50})_{\infty}}\sum_{n=-\infty}^{\infty}\frac{(-1)^nq^{75n^2+125n+49}}{1-q^{50n+35}}\no\\
		&+A_0+A_1+A_2+A_3+A_4,\label{paa}
	\end{align}
	where the second (resp. third) equality follows from Lemma \ref{lepp} (resp. Lemma \ref{leaa}). Extracting terms of the form $q^{5n+2}$ on both sides of \eqref{paa} yields
	\begin{align}
		&\sum_{n=0}^{\infty}(N^0(1,5,5n+2)-N^0(2,5,5n+2))q^{5n+2}=\frac{q^2[q^{25};q^{50}]_{\infty}(q^{50};q^{50})_{\infty}}{[q^{10},q^{20};q^{50}]_{\infty}}\no,
	\end{align}
	which gives \eqref{id52}.
\end{proof}
\section{Proof of Theorem \ref{thmain}}\label{s4}
Let 
\begin{align}
	\sum_{n=0}^\infty e(n)q^n:=\frac{\left(q^{5} ; q^{5}\right)^2_{\infty}} {\left(q^2 ; q^2\right)_{\infty}}\no.
\end{align}
Then from Theorem \ref{thmain2} we have
\begin{align*}
	N^0(1,5,5n+2)-N^0(2,5,5n+2)=e(n).
\end{align*}
Thus Theorem \ref{thmain} follows from
\begin{align}
	e(5^{\alpha}n+\lambda_\alpha)\equiv0 \pmod{5^{\left\lfloor\frac{\alpha+1}{2}\right\rfloor }}.\label{conen}
\end{align}
To prove \eqref{conen}, we define
\begin{align}
	L_{2\alpha -1}:&=\frac{\left(q^{10} ; q^{10}\right)_{\infty}}{\left(q ; q\right)^2_{\infty}}\sum_{n=0}^\infty e\left(5^{2\alpha -1}n+\lambda_{2\alpha-1}\right)q^{n+1}\label{del11}\\
	L_{2 \alpha}&:=\frac{\left(q^2 ; q^2\right)_{\infty}}{\left(q^{5} ; q^{5}\right)^2_{\infty}} \sum_{n=0}^{\infty} e\left(5^{2 \alpha} n+\lambda_{2 \alpha}\right) q^{n},\label{del22}
\end{align}
for $\alpha\geq1$ and $	L_0:=1.$
\begin{lemma}
	
	Then we have
	\begin{align}
		L_{2 \alpha+1} & =	U^{(0,0)}\left(L_{2 \alpha}\right), \label{l21}\\
		L_{2 \alpha+2} & =	U^{(1,0)}\left( L_{2 \alpha+1}\right),\label{l22}
	\end{align}
	for all $\alpha \geq 0$.
\end{lemma}
\begin{proof}
	We first prove \eqref{l21}.
	
	For any $\alpha \geq 0$, 
	\begin{align*}
		U^{(0,0)}\left(L_{2 \alpha}\right)& =U_5\left(\frac{q^2\left(q^{50} ; q^{50}\right)_{\infty}}{\left(q^{2} ; q^{2}\right)_{\infty}}\cdot\frac{\left(q^2 ; q^2\right)_{\infty}}{\left(q^{5} ; q^{5}\right)^2_{\infty}} \sum_{n=0}^{\infty} e\left(5^{2 \alpha} n+\lambda_{2 \alpha}\right) q^{n}\right) \\
		& =\frac{\left(q^{10} ; q^{10}\right)_{\infty}}{\left(q; q\right)^2_{\infty}} \cdot U_5\left(\sum^\infty_{n =0} e\left(5^{2 \alpha}n+\lambda_{2 \alpha}\right) q^{n+2}\right) \\
		& =\frac{\left(q^{10} ; q^{10}\right)_{\infty}}{\left(q; q\right)^2_{\infty}} \cdot \sum^\infty_{n =0} e\left(5^{2 \alpha}(5n+3)+\lambda_{2 \alpha}\right) q^{n+1}\\
		& =\frac{\left(q^{10} ; q^{10}\right)_{\infty}}{\left(q; q\right)^2_{\infty}} \cdot \sum^\infty_{n =0} e\left(5^{2 \alpha+1}n+\lambda_{2 \alpha+1}\right) q^{n+1} \\
		&=L_{2\alpha+1}.
	\end{align*}
	
	Next, we consider \eqref{l22}.
	
	We have
	\begin{align*}
		U^{(1,0)}\left( L_{2 \alpha+1}\right) & =U_5\left(\frac{\left(q ; q\right)^2_{\infty}}{q^2\left(q^{25} ; q^{25}\right)^2_{\infty}}\cdot\frac{\left(q^{10} ; q^{10}\right)_{\infty}}{\left(q ; q\right)^2_{\infty}}\sum_{n=0}^\infty e\left(5^{2\alpha +1}n+\lambda_{2\alpha+1}\right)q^{n+1}\right) \\
		& =\frac{\left(q^2 ; q^2\right)_{\infty}}{\left(q^5 ; q^5\right)^2_{\infty}} \cdot U_5\left(\sum^\infty_{n =0} e\left(5^{2 \alpha+1}n+\lambda_{2 \alpha+1}\right) q^{n-1}\right) \\
		& =\frac{\left(q^2 ; q^2\right)_{\infty}}{\left(q^5 ; q^5\right)^2_{\infty}} \cdot \sum^\infty_{n =0} e\left(5^{2 \alpha+1}(5n+1)+\lambda_{2 \alpha+1}\right) q^{n} \\
		& =\frac{\left(q^2 ; q^2\right)_{\infty}}{\left(q^5 ; q^5\right)^2_{\infty}} \cdot \sum^\infty_{n =0} e\left(5^{2 \alpha+2}n+\lambda_{2 \alpha+2}\right) q^{n} \\
		&=L_{2\alpha+2}.
	\end{align*}
\end{proof}

Denote the $5$-adic order of $n$ by $\pi(n)$ and set $\pi(0)=+\infty$. Then the following result holds.
\begin{theorem}\label{l2al}
	There exists discrete arrays $c, d$ such that for $\alpha\geq1$
	\begin{align}
		L_{\alpha}=\sum_{n\geq\delta_\alpha}c(\alpha,n)t^n+\rho\left(\sum_{n\geq\delta_\alpha}d(\alpha,n)t^n\right),
	\end{align}
	where \begin{align*}
		\delta_\alpha=\left\{
		\begin{aligned}
			&1,\quad\text{if}\quad \alpha \text{ is odd}\\
			&0,\quad\text{else}
		\end{aligned}
		\right..
	\end{align*}
	Moreover, we have
	\begin{align}
		\pi(c(\alpha,n))\geq\left\{
		\begin{aligned}
			\frac{\alpha-1}{2}&+\left\lfloor\frac{5 n-2}{3}\right\rfloor,\quad\text{if}\quad \alpha \text{ is odd}\\
			\frac{\alpha}{2}&+\left\lfloor\frac{5 n+1}{3}\right\rfloor,\quad\text{else}
		\end{aligned}
		\right.\label{pic2}
		\intertext{and}
		\pi(d(\alpha,n))\geq\left\{
		\begin{aligned}
			\frac{\alpha-1}{2}&+\left\lfloor\frac{5 n-1}{3}\right\rfloor,\quad\text{if}\quad \alpha \text{ is odd}\\
			\frac{\alpha}{2}&+\left\lfloor\frac{5 n+3}{3}\right\rfloor,\quad\text{else}
		\end{aligned}
		\right..\label{pid2}
	\end{align}
\end{theorem}

Note that congruences \eqref{conen} follow immediately from \eqref{del11}, \eqref{del22} and Theorem \ref{l2al}.
Thus, to complete the proof of Theorem \ref{thmain}, it suffices to show Theorem \ref{l2al}. To do this, we need two lemmas.
Recall \cite[Lemma 2.8]{wang}:
\begin{lemma}\label{leppi}
	Let $g(k, n)$ be integers which satisfy the recurrence relation \eqref{rr}. Suppose there exists an integer $l$ and a constant $\gamma$ such that for $l \leq k \leq l+4$ we have
	\begin{align}\label{pi}
		\pi(g(k, n)) \geq\left\lfloor\frac{5 n-k+\gamma}{3}\right\rfloor .
	\end{align}
	Then \eqref{pi} holds for any $k \in \mathbb{Z}$.
\end{lemma}
Using Lemma \ref{leppi}, we obtain
\begin{lemma}\label{lepi}
	For $n, k\in\mathbb{Z}$, we have
	\begin{align*}
		\pi(a_{0,0}(k, n)) &\geq\left\lfloor\frac{5 n-k-2}{3}\right\rfloor ,\\
		\pi(b_{0,0}(k, n)) &\geq\left\lfloor\frac{5 n-k-1}{3}\right\rfloor ,\\
		\pi(a_{0,1}(k, n)) &\geq\left\lfloor\frac{5 n-k-3}{3}\right\rfloor ,\\
		\pi(b_{0,1}(k, n)) &\geq\left\lfloor\frac{5 n-k}{3}\right\rfloor ,\\
		\pi(a_{1,0}(k, n)) &\geq\left\lfloor\frac{5 n-k+2}{3}\right\rfloor ,\\
		\pi(b_{1,0}(k, n)) &\geq\left\lfloor\frac{5 n-k+5}{3}\right\rfloor ,\\
		\pi(a_{1,1}(k, n)) &\geq\left\lfloor\frac{5 n-k+2}{3}\right\rfloor ,\\
		\pi(b_{1,1}(k, n)) &\geq\left\lfloor\frac{5 n-k+4}{3}\right\rfloor .
	\end{align*}
\end{lemma}
\begin{proof}
	One can verify these inequalities for $-4\leq k\leq0$ using Group I-IV in Lemma \ref{legroup}.
	Then apply Lemma \ref{leppi} to complete the proof.
\end{proof}
Now we are in a position to prove Theorem \ref{l2al}.
\begin{proof}[Proof of Theorem \ref{l2al}]
	Define the discrete arrays $c, d$ inductively.
	
	Let
	$c(1,1)=5, c(1,2)=5^2, c(1,k)=0$ with $k\geq3$ and  $d(1,1)=-5, d(1,k)=0$ with $k\geq2$. For $\alpha\geq1$, define
	\begin{align*}
		c(2\alpha,n)&=\sum_{k\geq1}\left[c(2\alpha-1,k)a_{1,0}(k,n)+d(2\alpha-1,k)a_{1,1}(k,n)\right],\\
		d(2\alpha,n)&=\sum_{k\geq1}\left[c(2\alpha-1,k)b_{1,0}(k,n)+d(2\alpha-1,k)b_{1,1}(k,n)\right],\\
		c(2\alpha+1,n)&=\sum_{k\geq0}\left[c(2\alpha,k)a_{0,0}(k,n)+d(2\alpha,k)a_{0,1}(k,n)\right],\\
		d(2\alpha+1,n)&=\sum_{k\geq0}\left[c(2\alpha,k)b_{0,0}(k,n)+d(2\alpha,k)b_{0,1}(k,n)\right].
	\end{align*}
	
	By Lemma \ref{legroup}, Group I and \eqref{l21}, we have
	\begin{align}
		L_1=U^{(0,0)}(L_0)=U^{(0,0)}(1)=\sum_{n=1}^2c(1,n)t^n+\rho\left(\sum_{n=1}^1d(1,n)t^n\right)\no.
	\end{align}
	Obviously, $\pi(c(1,n))\geq\left\lfloor\frac{5 n-2}{3}\right\rfloor $ and  $\pi(d(1,n))\geq\left\lfloor\frac{5 n-1}{3}\right\rfloor $.
	
	Next, we consider $L_2.$ Using \eqref{l22}, we find that
	\begin{align*}
		L_2&=U^{(1,0)}(L_1)=U_5(H\cdot L_1)
		\no		\\&	=\sum_{n=1}^2c(1,n)
		U_5(H t^n)+\sum_{n=1}^1d(1,n)U_5(H\rho t^n)
		\no	\\&=\sum_{n=1}^2c(1,n)
		U^{(1,0)}( t^n)+\sum_{n=1}^1d(1,n)U^{(1,1)}( t^n)
		\no					\\&=\sum_{n=1}^2c(1,n)
		\left(\sum_{k\geq\left\lceil\frac{n-4}{5}\right\rceil}a_{1,0}(n,k)t^k+\rho\left(\sum_{k\geq\left\lceil\frac{n-4}{5}\right\rceil}b_{1,0}(n,k)t^k\right)\right)\no\\&\quad+\sum_{n=1}^1d(1,n)	\left(\sum_{k\geq\left\lceil\frac{n-3}{5}\right\rceil}a_{1,1}(n,k)t^k+\rho\left(\sum_{k\geq\left\lceil\frac{n-3}{5}\right\rceil}b_{1,1}(n,k)t^k\right)\right)\qquad\qquad\textrm{(by \eqref{uab})}
		\\&=\sum_{n\geq0}\left[\left(\sum_{k=1}^2c(1,k)a_{1,0}(k,n)\right)+\sum_{k=1}^1d(1,k)a_{1,1}(k,n)\right]t^n\\&\quad+
		\rho\left\{\sum_{n\geq0}\left[\left(\sum_{k=1}^2c(1,k)b_{1,0}(k,n)\right)+\sum_{k=1}^1d(1,k)b_{1,1}(k,n)\right]t^n\right\}
		\\&=\sum_{n\geq0}c(2,n)t^n+\rho\left(\sum_{n\geq0}d(2,n)t^n\right).
	\end{align*}
	Use Lemma \ref{lepi}
	\begin{align*}
		\pi(c(2,n))&\geq\min\big\{\pi(c(1,1))+\pi(a_{1,0}(1,n)),\pi(c(1,2))+\pi(a_{1,0}(2,n)),\\&\qquad\qquad\pi(d(1,1))+\pi(a_{1,1}(1,n)) \big\}\\&\geq1+\left\lfloor\frac{5 n+1}{3}\right\rfloor ,
		\\
		\pi(d(2,n))&\geq\min\big\{\pi(c(1,1))+\pi(b_{1,0}(1,n)),\pi(c(1,2))+\pi(b_{1,0}(2,n)),\\&\qquad\qquad\pi(d(1,1))+\pi(b_{1,1}(1,n)) \big\}\\&\geq1+\left\lfloor\frac{5 n+3}{3}\right\rfloor .
	\end{align*}	
	Thus the result holds for $L_{\alpha}$ when $\alpha=1, 2$. We proceed by induction.
	Suppose that the result hold for $L_{2\alpha}$. Then apply \eqref{uab} and \eqref{l21}.
	\begin{align}
		L_{2\alpha+1}&=U^{(0,0)}(L_{2\alpha})=U_5(Z\cdot L_{2\alpha})\no\\
		&=\sum_{n\geq0}c(2\alpha,n)
		U_5(Z t^n)+\sum_{n\geq0}d(2\alpha,n)U_5(Z\rho t^n)\no
		\\
		&=\sum_{n\geq0}c(2\alpha,n)\no
		\no	U^{(0,0)}( t^n)+\sum_{n\geq0}d(2\alpha,n)	U^{(0,1)}(t^n)\\
		&=\sum_{n\geq0}c(2\alpha,n)
		\left(\sum_{k=\left\lceil\frac{n+1}{5}\right\rceil}^{\infty} a_{0,0}(n, k) t^k+\rho\left(\sum_{k=\left\lceil\frac{n+1}{5}\right\rceil}^{\infty} b_{0,0}(n, k) t^k\right)\right)\no
		\\ &\quad+\sum_{n\geq0}d(2\alpha,n)
		\left(\sum_{k=\left\lceil\frac{n+2}{5}\right\rceil}^{\infty} a_{0,1}(n, k) t^k+\rho\left(\sum_{k=\left\lceil\frac{n+2}{5}\right\rceil}^{\infty} b_{0,1}(n, k) t^k\right)\right)\no
		\\
		&=\sum_{n\geq1}
		\left[\sum_{k=0}^{\infty}(c(2\alpha,k) a_{0,0}(k, n) +d(2\alpha,k) a_{0,1}(k, n))\right]t^n\no
		\\&\quad+\rho\left\{\sum_{n\geq1}\no
		\left[\sum_{k=0}^{\infty}(c(2\alpha,k) b_{0,0}(k, n) +d(2\alpha,k) b_{0,1}(k, n))\right]t^n\right\}\\
		&=\sum_{n\geq1}c(2\alpha+1,n)t^n+\rho\left(\sum_{n\geq1}d(2\alpha+1,n)t^n\right).
	\end{align}
	Moreover, using Lemma \ref{lepi}, \eqref{pic2} and \eqref{pid2}, we find that
	\begin{align}
		\pi(c(2\alpha+1,n))&\geq\min_{k\geq0}\Big\{\pi(c(2\alpha,k))+\pi( a_{0,0}(k, n)),\no\\&\qquad\qquad\pi(d(2\alpha,k))+\pi( a_{0,1}(k, n))\Big\}\no\\
		&\geq\alpha+\left\lfloor\frac{5 n-2}{3}\right\rfloor \no
		\intertext{and}
		\pi(d(2\alpha+1,n))&\geq\min_{k\geq0}\Big\{\pi(c(2\alpha,k))+\pi( b_{0,0}(k, n)),\no\\&\qquad\qquad\pi(d(2\alpha,k))+\pi( b_{0,1}(k, n))\Big\}\no\\
		&\geq\alpha+\left\lfloor\frac{5 n-1}{3}\right\rfloor \no.
	\end{align}
	Thus the result holds for $L_{2\alpha+1}$. 
	
	Next, use \eqref{uab} and \eqref{l22} to obtain
	\begin{align}
		L_{2\alpha+2}&=U^{(1,0)}(L_{2\alpha+1})=U_5(H\cdot L_{2\alpha+1})\no\\
		&=\sum_{n\geq1}c(2\alpha+1,n)
		U_5(H t^n)+\sum_{n\geq1}d(2\alpha+1,n)U_5(H\rho t^n)\no
		\\
		&=\sum_{n\geq1}c(2\alpha+1,n)\no
		\no	U^{(1,0)}( t^n)+\sum_{n\geq1}d(2\alpha+1,n)	U^{(1,1)}(t^n)\\
		&=\sum_{n\geq1}c(2\alpha+1,n)
		\left(\sum_{k=\left\lceil\frac{n-4}{5}\right\rceil}^{\infty} a_{1,0}(n, k) t^k+\rho\left(\sum_{k=\left\lceil\frac{n-4}{5}\right\rceil}^{\infty} b_{1,0}(n, k) t^k\right)\right)\no
		\\ &\quad+\sum_{n\geq1}d(2\alpha+1,n)
		\left(\sum_{k=\left\lceil\frac{n-3}{5}\right\rceil}^{\infty} a_{1,1}(n, k) t^k+\rho\left(\sum_{k=\left\lceil\frac{n-3}{5}\right\rceil}^{\infty} b_{1,1}(n, k) t^k\right)\right)\no
		\\
		&=\sum_{n\geq0}
		\left[\sum_{k=1}^{\infty}(c(2\alpha+1,k) a_{1,0}(k, n) +d(2\alpha+1,k) a_{1,1}(k, n))\right]t^n\no
		\\&\quad+\rho\left\{\sum_{n\geq0}\no
		\left[\sum_{k=1}^{\infty}(c(2\alpha+1,k) b_{1,0}(k, n) +d(2\alpha+1,k) b_{1,1}(k, n))\right]t^n\right\}\\
		&=\sum_{n\geq0}c(2\alpha+2,n)t^n+\rho\left(\sum_{n\geq0}d(2\alpha+2,n)t^n\right).
	\end{align}
	Then again by Lemma \ref{lepi}, \eqref{pic2} and \eqref{pid2} we have
	\begin{align}
		\pi(c(2\alpha+2,n))&\geq\min_{k\geq1}\Big\{\pi(c(2\alpha+1,k))+\pi( a_{1,0}(k, n)),\no\\&\qquad\qquad\pi(d(2\alpha+1,k))+\pi( a_{1,1}(k, n))\Big\}\no\\
		&\geq\alpha+1+\left\lfloor\frac{5 n+1}{3}\right\rfloor \no
		\intertext{and}
		\pi(d(2\alpha+2,n))&\geq\min_{k\geq1}\Big\{\pi(c(2\alpha+1,k))+\pi( b_{1,0}(k, n)),\no\\&\qquad\qquad\pi(d(2\alpha+1,k))+\pi( b_{1,1}(k, n))\Big\}\no\\
		&\geq\alpha+1+\left\lfloor\frac{5 n+3}{3}\right\rfloor \no.
	\end{align}
	Thus the result holds for $L_{2\alpha}$. This proves Theorem \ref{l2al} by induction.
\end{proof}

\noindent {\bf Acknowledgements}
This work is supported by NSFC (National Natural Science
Foundation of China) through Grants NSFC 12071331.

\bigskip

\noindent {\bf Data Availability}
All data generated during this study are included in the published article.

\bigskip

\noindent {\bf  Declarations}

\noindent {\bf   Conflict of interest}
The authors declare that they have no known competing financial interests or personal
relationships that could have appeared to influence the work reported in this paper.

\end{document}